\documentclass{birkmult}
 \newtheorem{theorem}{Theorem}[section]
 \newtheorem{corollary}[theorem]{Corollary}
 
 \newtheorem{proposition}[theorem]{Proposition}
 \theoremstyle{definition}
 \newtheorem{definition}[theorem]{Definition}
 \theoremstyle{remark}
 \newtheorem{remark}[theorem]{Remark}
 \newtheorem{example}{Example}
 \numberwithin{equation}{section}

\DeclareMathOperator{\Res}{\mathcal{R}}

\def\R#1{\mathbb{R}^{#1}}
\def\Com#1{\mathbb{C}^{#1}}

\def\I{\mathrm{i}}

\def\R{{\mathbb R}}

\begin{document}

\title{On the Jacobian of the harmonic moment map}

\author{Bj\"orn Gustafsson}

\address{Department of Mathematics\\
Royal Institute of Technology\\
S-100 44 Stockholm\\
Sweden}

\email{gbjorn@kth.se}

\author{Vladimir Tkachev}

\address{Department of Mathematics\\
Royal Institute of Technology\\
S-100 44 Stockholm\\
Sweden}

\email{tkatchev@kth.se}



\begin{abstract}
In this paper we represent harmonic moments in the language of transfinite
functions, that is projective limits of polynomials in infinitely
many variables. We obtain also an explicit formula for the Jacobian of a generalized harmonic moment map.
\end{abstract}

\thanks{The authors have been supported by supported by the Swedish Research Council,
grant KAW 2005.0098 from the Knut and Alice Wallenberg Foundation,
and by the European Science Foundation Networking Programme HCAA}

\maketitle



\section{Introduction}

With any integer $k$ and any closed oriented analytic curve $\Gamma$
one can associate the $k$th harmonic moment of $\Gamma$, defined as
$$
M_k(\Gamma)=\frac{1}{2\pi \I}\int_{\Gamma}z^k\bar z\; dz.
$$
When $k$ is nonnegative and $\Gamma=\partial \Omega$ for some domain
$\Omega$, the moment takes the more familiar form
$$
M_k(\Gamma)=\frac{1}{\pi}\iint_{\Omega}z^k dxdy, \qquad z=x+\I y.
$$
Information about $\Gamma$ can be read off from these harmonic
moments, which have turned out to be  useful geometric objects
in many problems of complex analysis and potential theory.

The harmonic moments appear also in algebraic contexts, in the first
place in connection with conservation laws discovered in 40's by
P.~Polubarinova-Kochina \cite{PK} and L.~Galin \cite{Gal}, and were
also studied in 70's by S.~Richardson \cite{R72} in application to
the Hele-Shaw problem (see \cite{GVas} for a full account of
relevant material). It turns out that the quantities $M_k(\Gamma)$
constitute a hierarchy of conservative quantities for the boundary
$\Gamma$ under the action of Hele-Shaw flow with a source at the
origin (Laplacian growth). This allows to describe Hele-Shaw
evolution explicitly for a wide class of polynomial domains. This
complete integrability has recently been subject to new
investigations by I.~Krichever, A.~Marshakov, M.~Mineev-Weinstein,
P.~Wiegmann and A.~Zabrodin, e.g. \cite{WZ}, \cite{KKMWZ},
\cite{KMWZ}. One point of view is that harmonic moments can be
thought of as canonical coordinates in certain well established
integrable hierarchies (for example, in the dispersionless 2D Toda
hierarchy), and it has for example been shown that they can be
written as derivatives of an associated tau-function.

Despite the above mentioned applications, some principal
difficulties remain on the level of mathematical rigor with the
basic definitions of harmonic moments regarded as functionals in
infinitely many variables. In particular, there is no established
algebraic or analytic calculus which allows importing harmonic
moments as well-defined functionals. In this paper we make an
attempt to represent harmonic moments in the language of transfinite
functions, that is projective limits of polynomials in infinitely
many variables. In this picture, not only the usual complex
variables ($z$ etc.), but also the coefficients ($a_0, a_1,\dots$)
of analytic functions are treated as variables. Some analogies with
symmetric functions or germs of analytic functions may be traced.
For example, the number of variables is irrelevant in symmetric
functions and any symmetric function is uniquely determined if the
number of variables is large enough. A similar rigidity is valid for
harmonic moments in the sense that they stabilize after truncation
of variables of higher grade.

We mention that although the model of transfinite functions
discussed below allows manipulation of objects in an algebraic
manner, it does not allow a priori to `evaluate' the objects. In
some particular cases, for example for domains which are conformal
images of univalent polynomials, the transfinite calculus becomes
finite, which makes direct evaluations possible. In more involved
cases one needs an adequate homomorphism into one of standard
evaluation rings. However, we will not pursue this matter in the
present article.

Another, more concrete, application of the above formalism is an
explicit formula for the Jacobian of a generalized harmonic moment map.


\section{Transfinite functions}

\subsection{Projective limits of polynomial rings}

Let $R$ be a commutative ring with unit and $A$ a set of independent
commutating\footnote{In the paper we deal only with commutating
variables while all the constructions below are still valid for
non-commutative setup without any changes.} variables. Then the
polynomial ring $R[A]$ makes sense, even if the set $A$ is infinite
(of any cardinality). We shall introduce transfinite functions over
$A$ as formal sums (in general infinite) of monomials in the
variables $A$ with coefficients in $R$ such that for each finite
subset $F\subset A$ there are only finitely many terms which contain
variables only from $F$. In other words, by setting every variable
in $A\setminus F$ equal to zero a transfinite function reduces to a
polynomial in $R[F]$.

For the formal definition the notion of projective limit is
appropriate. Let $\mathcal{F}$ denote the family of finite subsets
of $A$. This is a directed set in a natural way: it is partially
ordered by inclusion ($F_1\subset F_2$, $F_1,F_2\in\mathcal{F}$) and
with this partial order any two elements have an upper bound, namely
their union (if $F_1,F_2\in\mathcal{F}$ then $F=F_1\cup
F_2\in\mathcal{F}$ and $F_1\subset F$, $F_2\subset F$). These
inclusions induce projection homomorphisms
\begin{equation}\label{eq:projection}
\pi_{F_1, F_2}: R[A]/(A\setminus F_2)\to  R[A]/(A\setminus F_1),
\end{equation}
where, for any subset  $S\subset R[A]$, $(S)$ denotes the ideal in
$R[A]$ generated by $S$.

The maps (\ref{eq:projection}) define a projective system of rings
based on the directed set $\mathcal{F}$. We define transfinite
functions by passing to the projective (or inverse) limit.

\begin{definition}
The ring of \textit{transfinite} functions over $A$  and with coefficients in $R$ is the
projective limit
$$
R_\infty [A] =\lim_{\leftarrow{F}} R[A]/(A\setminus F).
$$
\end{definition}

Recall that, as a set, the projective limit $R_\infty[A]$ can be
taken to be that subset of the cartesian product
$\Pi_{F\in\mathcal{F}} R[A]/(A\setminus F)$ for which the $F_1$'s
and $F_2$'s components are related by the map $\pi_{F_2, F_1}$
whenever $F_1\subset F_2$. It is actually enough to consider only
cofinal segments in the above cartesian product, because for any
$F\in\mathcal{F}$, the $F$'s component of an element determines
uniquely the $F_1$'s, for any $F_1\subset F$. In some operations
with the projective limit the so arising possibility of
self-correction of initial segments is important. One may think of
$R_\infty[A]$ as a kind of completion of $R[A]$.

The above definition of transfinite functions is modeled on standard
definitions of $p$-adic numbers and formal power series. For
example, the ring of formal power series is
$\mathbb{C}[[z]]=\lim_{\leftarrow n}\mathbb{C}[z]/(z^n)$. However,
the transfinite functions are actually somewhat simpler, because for
any $F\in \mathcal{F}$ there is a natural embedding $R[F]\to R[A]$
which becomes an isomorphism
$$
R[F]\cong R[A]/(A\setminus F).
$$
The projection maps (\ref{eq:projection}) therefore have the
alternative description as maps
\begin{equation}\label{eq:projection1}
\pi_{F_1,F_2}: R[F_2]\to R[F_1],
\end{equation}
most easily described by declaring that all variables in
$F_2\setminus F_1$ shall be set equal to zero.

The natural projection maps
$$
\pi_F: R_\infty [A]\to R[A]/(A\setminus F).
$$
correspond in this simpler picture to the previously mentioned maps
$\pi_F: R_\infty [A]\to R[F]$ (setting all variables not in $F$
equal to zero). In the other direction we have inclusion maps
$R[F]\to R_\infty [A]$, which postcomposed with the $\pi_F$ give the
identity maps on the $R[F]$.

The definition with projection maps (\ref{eq:projection}) is
somewhat more flexible than (\ref{eq:projection1}) in that it allows
for including formal power series (with coefficients in $R$) into
the picture, namely by admitting powers $a^n$ ($a\in A$) among the
generators of the ideals. However, we shall not go into such
generalizations in the present paper, and therefore we shall work in
the simpler setting (\ref{eq:projection1}).


\subsection{Transfinite functions}

The typical way we shall use the above general definitions is as
follows. We are interested in forming transfinite functions in the
variables $z$, $a_0, a_1, a_2\dots$, plus sometimes corresponding
conjugated variables ($\bar{a}_1$ etc.) and variables with negative
index ($a_{-1}$ etc.), all treated as independent variables. We also
want to be able to invert a few of the variables, namely $z$, $a_0$,
i.e. admit $z^{-1}$ and $a_0^{-1}$. Hence these cannot be set equal
to zero. We can achieve the above goals by choosing the ground ring
$R$ to contain the variables we want to invert, namely we take
$$
R=\mathbb{C}[z,z^{-1},a_0, a_0^{-1}] :=\mathbb{C}[z,w,a_0,
b_0]/(zw-1, a_0 b_0 -1).
$$
Then our ring of transfinite functions will be
$$
\mathbb{C}_\infty [z,z^{-1},a_0, a_0^{-1};a_1, a_2, \dots] :=
R_\infty [a_1, a_2, \dots].
$$

Set
$$
A_0 =\{z,z^{-1},a_0, a_0^{-1}\}
$$
for the ``\textit{null-variables}", or invertible variables, and
$$
A_n =\{a_1, a_2,\dots, a_n\}
$$
for the remaining variables up to some index $n\geq 1$. This set
could, depending on the context, also contain the corresponding
conjugated variables (considered as independent variables), or
variables with negative index, for example $ A_n =\{a_1, \bar{a}_1,
a_2, \bar{a}_2 ,\dots, a_n, \bar{a}_n,\}$ or $ A_n =\{a_{-n},\dots, a_{-2},
a_{-1}, a_1, a_2,\dots, a_n\}$. Finally, set
$$
A_+ =\bigcup_{n\geq 1} A_n, \quad A= A_0\cup A_+.
$$

The ring of transfinite functions will in this context be denoted
$$
\mathbb{C}_\infty[A]=\mathbb{C}_\infty [z,z^{-1},a_0, a_0^{-1};a_1,
a_2,a_3, \dots]=R_\infty[A_+].
$$
In the construction of $\mathbb{C}_\infty[A]$ as a projective limit
it is enough to use the sets $A_n$ (in place of all finite subsets
of $A$). Thus, setting
$$
\mathbb{C}_n[A] =\mathbb{C} [z,z^{-1},a_0, a_0^{-1};a_1, a_2, \dots,
a_n] = R [A_n],
$$
we have the projective system $\{\mathbb{C}_n[A]\}_{n\geq 0}$, with
projection homomorphisms
$$
\pi_{ij}: \mathbb{C}_{j}[A]\to \mathbb{C}_{i}[A], \qquad j\geq i\geq
0
$$
defined by setting the variables in $A_j\setminus A_i$ equal to zero
(equivalently, by removing every term which contains such a
variable). The projective limit is
$$
\mathbb{C}_\infty[A]=\lim_{\leftarrow n}\mathbb{C}_n[A]
$$
with projections
$$
\pi_n: \mathbb{C}_\infty[A]\to\mathbb{C}_n[A],
$$
obtained by setting the variables in $A_+\setminus A_n$ equal to
zero.

In case conjugate variables are included there is a natural
involution
$$
{}^* : \mathbb{C}_\infty[A]\to \mathbb{C}_\infty[A]
$$
which for any variable which has a conjugate exchanges the two
(e.g., $a_1\mapsto \bar{a}_1$, $\bar{a}_1\mapsto a_1$). Variables
which do not have a conjugate should be thought of as real
variables, and remain unchanged under the involution.

One way to describe the ring of transfinite functions is to consider
all sequences $(h_n)_{n\geq 0}$ of elements in $\mathbb{C}_{n}[A]$
stable under high-order substitutions
\begin{equation}\label{restr}
\pi_{ij}(h_j)=h_{i}, \qquad j\geq i\geq 0,
\end{equation}
and equipped with pointwise algebraic operations. In what follows we
make no distinction between such sequences and their limits in
$\mathbb{C}_{\infty}[A]$ and write $h=\lim\limits_{\leftarrow}h_n$.
The element $h_n$ is called the $n$th approximant of $h$. The
described (projective) convergence is rigid in the sense that any
approximant $h_n$ is determined uniquely by $h$.

In the other direction we have the injection $\mathbb{C}_n[A]\to
\mathbb{C}_\infty[A]$, by which any element $f\in \mathbb{C}_{n}[A]$
gives rise to a transfinite function by taking the inverse limit,
$\lim\limits_{\leftarrow}f_k$, of the sequence
\begin{equation*}\label{complet}
f_k:= \left\{
  \begin{array}{ll}
   \pi_{kn}( f), & \hbox{\text{for
$0\le k\le n$;}} \\
    f , & \hbox{\text{for $k\geq n$}}
  \end{array}
\right.
\end{equation*}
This $\lim\limits_{\leftarrow}f_k$ is the unique element in
$\mathbb{C}_\infty[A]$ with the property that its $n$th approximant
is exactly $f$. In what follows we identify `finite' polynomials
from $\mathbb{C}_n[A]$ and their lifting in $\mathbb{C}_\infty[A]$
(that is consider the $\mathbb{C}_n[A]$ as subrings in
$\mathbb{C}_\infty[A]$).

It is natural consider a transfinite function  $h$ as a function in
an infinite number  of variables. Although we think of all the
variables $A$ as `complex' variables and `variable coefficients', no
value (`number') can be assigned to  $h$ at any concrete point.
However, one is allowed to substitute finitely many variables by
complex numbers (this is a special case of operation 2) in the next
subsection), and the result will be another transfinite function.

In order to pass completely from transfinite objects to classical
ones, however, one need to have an evaluation on
$\mathbb{C}_\infty[A]$, like the limit in topological categories.
Algebraically this is equivalent to constructing an adequate
homomorphism from the ring $\mathbb{C}_{\infty}[A]$ to some standard
evaluation  ring. A choice of such a homomorphism should be made
individually in each concrete case.

\begin{example}\label{ex:1}
The simplest example of a transfinite function is the
\textit{transfinite power series} in $z$, with variable coefficients
$a_k$:
\begin{equation}\label{pol}
\lim_\leftarrow\sum_{k=0}^n a_kz^{k+1}=\sum_{k=0}^\infty a_kz^{k+1}.
\end{equation}
This is an element in $\mathbb{C}_\infty[z,a_0,a_1,\dots]$. On the
other hand, a power series with constant coefficients, like
$\sum_{k=0}^\infty k! z^{k}$, is not a transfinite function in our
sense, but is an element of $\mathbb{C}[[z]]=\lim_{\leftarrow
n}\mathbb{C}[z]/(z^n)$.
\end{example}


\subsection{Operations with transfinite functions}
\label{sec:index}

{\bf 1)} ({\it General maps.}) Consider a map $q:\mathbb{C}[A]\to
\mathbb{C}[A]$ which commutes with restrictions for large enough
indices:
\begin{equation*}\label{gradingmap}
q\circ \pi_{ij}=\pi_{ij}\circ q, \qquad j\ge i\geq m,
\end{equation*}
where $m\geq 0$ depends only on $q$. Then $q$ extends naturally to a
map on $\mathbb{C}_{\infty}[A]$ as follows. Let $x\in \mathbb{C}_{\infty}[A]$
and $x_n\in \mathbb{C}_{n}[A]$ be  the sequence of its approximants.
Define $y_i=q(x_i)$ for $i\geq m$ and $y_i=\pi_{im}(y_m)$ for $i\leq m$. Then
$(y_n)_{n\geq 0}$ satisfies (\ref{restr}) because for any $j\geq i\geq m$
$$
\pi_{ij}(y_j)=\pi_{ij}(q(x_j))=q(\pi_{ij}(x_j))=q(x_i)=y_i
$$
and similarly one checks (\ref{restr}) for small indices.
Therefore $(y_n)_{n\geq 0}$ induces en element in $\mathbb{C}_{\infty}[A]$
denoted by $q(x)$.

{\bf 2)} ({\it Substitution.}) Let $A$, $B$ be two families of
variables, filtered by $A_n$, $B_n$ ($n\geq 0$) respectively,
let $\phi\in \mathbb{C}_{\infty}[B]$, let $X\subset B$ be a finite
subset and let $t:X\to \mathbb{C}_{\infty}[A]$ be any map. Then for
any $i\geq 0$ we define
$$
\psi_i=\phi_i|_{(t,X)}\in \mathbb{C}_{i}[A,B\setminus X],
$$
where $|_{(t,X)}$ means that one makes the substitutions
$x=(t(x))_n=\pi_n \circ t(x)\in \mathbb{C}_n[A]$ for each $x\in X$.
This new sequence obviously satisfies (\ref{restr}), hence induces
an element in $\mathbb{C}_{\infty}[A,B\setminus X]$, which we denote
by $\phi|_{(t,X)}$.

\begin{remark}
The introduced composition of two transfinite functions is very
close to what is known as the plethysm in category of symmetric
functions (see, e.g., \cite[p.~135]{MacDon}.
\end{remark}

{\bf 3)} ({\it Derivation.}) Another example is the partial
derivative. For any variable $x\in A_m$
$$
\partial_{x}\circ \pi_{ij}=\pi_{ij}\circ \partial_{x}, \qquad i\geq m.
$$
Therefore $\partial_x$ extends to $\partial_{x}:
\mathbb{C}_{\infty}[A]\to \mathbb{C}_{\infty}[A]$. One then checks
easily that $\partial_x$ is a derivation on
$\mathbb{C}_{\infty}[A]$, that is $\partial_{x}$ is linear and
satisfies the Leibniz rule. Moreover, the derivatives satisfy the
usual commutativity: for any $x,y$
$$
\partial_x\partial_y=\partial_y\partial_x.
$$

{\bf 4)} ({\it Coefficient extraction.}) An important operation
is coefficient extraction. Let $y\in \mathbb{C}_{\infty}[A]$ and
$x\in A$. Then for any fixed $n$ the coefficient
$[x^n](y_i)$ of $x^n$ in $y_i\in \mathbb{C}_{i}[A]$ is well defined
and extends in an obvious way to an element $[x^n](y)\in
\mathbb{C}_{\infty}[A]$.


\subsection{The transfinite resultant}

Recall that the resultant  of two  polynomials
$f(z)=a_n\prod_{i=1}^n(z-\xi_i)=\sum_{i=0}^n a_iz^i$
and
$g(z)=b_m\prod_{j=1}^m(z-\eta_j) =\sum_{j=0}^m b_jz^j$
is a polynomial function in the coefficients of $f$ and $g$
having the elimination property that it vanishes if and only if $f$ and $g$ have
a common zero \cite{Waerden}, \cite{Gelfand-Kapranov-Zelevinskij}. In terms of the zeros of the polynomials
the resultant is given by the Poisson product formula
\begin{equation*}\label{res1}
\begin{split}
\Res_{\mathrm{pol}}(f,g)&=a_n^mb_m^n\prod_{i,j} (\xi_i-\eta_j) =a_n^m\prod_{i=1}^n g(\xi_i).
\end{split}
\end{equation*}
Alternatively, the resultant can be computed as the determinant of the Sylvester matrix of size $n+m$
\begin{equation*}\label{sfor}
\Res_{\mathrm{pol}}(f,g)=\det
\begin{pmatrix}
a_0   & a_1  &  \ldots & a_n \\
     & \ldots &\ldots &  \ldots & \ldots \\
    &  &a_0  & a_1 & \ldots & a_n   \\
   b_0  & b_{1} & \ldots & b_m \\
      & \ldots &\ldots &  \ldots & \ldots \\
     &  &  b_0 & b_{1} &\ldots & b_m  \\
\end{pmatrix},
\end{equation*}
where the first $m$ rows are the shifted coefficients of $f$, the next
$n$ rows are the shifted coefficients of $g$.

The authors introduced recently \cite{GT07} a notion of the meromorphic resultant of two meromorphic functions on an arbitrary compact Riemann surface. For any two meromorphic functions $f$ and $g$, whose divisors and have no common points, the number
\begin{equation*}
\label{gagb1new}
\Res(f,g)=g((f))\equiv \prod_{i} g(\xi_i)^{N_i},
\end{equation*}
is called the meromorphic resultant of $f$ and $g$. Here $(f)=\sum_{i}N_i\cdot \xi_i$ is the divisor of $f$. For the general properties of the meromorphic resultant, see \cite{GT07}. We mention only that the meromorphic resultant is symmetric and homogeneous of degree zero (i.e. depends only on the divisors of $f$ and $g$). Moreover, for two rational functions $f(z)=\sum_{k=0}^n b_kz^{-k}$  and
$ g(z)=\sum_{k=0}^m a_kz^{k}$, one easily finds that their meromorphic resultant is related to the classical polynomial resultant by the formula:
\begin{equation}\label{formula}
\begin{split}
\Res(\sum_{k=0}^n b_kz^{-k}&,\sum_{k=0}^m a_kz^{k})=\frac{1}{a_0^nb_0^m}\Res_{\mathrm{pol}}(\sum_{k=0}^n b_kz^{n-k},\sum_{k=0}^m a_kz^{k})\\
&=\frac{1}{a_0^nb_0^m}\det
\begin{pmatrix}
b_0   & b_1  &  \ldots & b_n \\
     & \ldots &\ldots &  \ldots & \ldots \\
    &  &b_0  & b_1 & \ldots & b_n   \\
   a_n  & a_{n-1} & \ldots & a_0 \\
      & \ldots &\ldots &  \ldots & \ldots \\
     &  &  a_n & a_{n-1} &\ldots & a_0  \\
\end{pmatrix}.
\end{split}
\end{equation}

What we understand by the transfinite resultant is actually the inverse limit of the latter meromorphic resultant. Indeed, one can check that for $n\geq 1$ the  Sylvester's determinants
\begin{equation*}
\begin{split}
\Res(a,b)_n&=a_0^{-n}b_0^{-n}
\det
\begin{pmatrix}
b_0   & b_1  &  \ldots & b_n \\
     & \ldots &\ldots &  \ldots & \ldots \\
    &  &b_0  & b_1 & \ldots & b_n   \\
   a_n  & a_{n-1} & \ldots & a_0 \\
      & \ldots &\ldots &  \ldots & \ldots \\
     &  &  a_n & a_{n-1} &\ldots & a_0  \\
\end{pmatrix}\\
&\equiv \Res(\sum_{k=0}^n b_kz^{-k},\sum_{k=0}^n a_kz^k)=\Res(\sum_{k=0}^n a_kz^k,\sum_{k=0}^n b_kz^{-k})
\end{split}
\end{equation*}
satisfy the transfinite property (\ref{restr}): substitution $a_n=b_n=0$ into the above determinant
gives $\Res(a,b)_{n-1}$ (recall also that the meromorphic resultant is a symmetric function of its arguments). This shows that the sequence $(\Res(a,b)_n)_{n\geq 0}$ generates an element in $\mathbb{C}_{\infty}[a_0,a_0^{-1},b_0,b_0^{-1};a_1,b_1,\dots]$ denoted by $\Res(a,b)$ and called the \textit{transfinite resultant}.
Its two initial approximants are:
\begin{equation*}
\begin{split}
\Res(a,b)_1&=1-\frac{a_1{b}_1}{a_0b_0},\\
\Res(a,b)_2&=1-\frac{a_1{b}_1+2a_2{b}_2}{a_0b_0}-
\frac{a_0a_2 {b}^2_1+b_0a_1^2{b}_2-a_2^2{b}_2^2+a_1a_2{b}_1{b}_2}{a_0^2b_0^2}.
\end{split}
\end{equation*}

\subsection{The transfinite elimination function}
In many applications the so-called elimination function is more advantageous than the meromorphic resultant. It is defined by
$$
\mathcal{E}_{f,g}(u,v):=\Res(f-u,g-v),
$$
where $u$ and $v$ are two independent complex variables. It is  known (see \cite{GT07}) that for any two meromorphic functions $f$ and $g$ on a closed Riemann surface this function is rational and satisfies the following elimination property:
$$
\mathcal{E}_{f,g}(f(\zeta),g(\zeta))\equiv 0.
$$
The transfinite version of the elimination function is defined by
\begin{equation*}\label{Syl}
\begin{split}
\mathcal{E}_{a,b}(u,v)&=\lim_{\leftarrow n}\Res(-u+\sum_{k=0}^n a_kz^{k+1},-v+\sum_{k=0}^n b_kz^{-k-1}).
\end{split}
\end{equation*}
The latter transfinite function can be also viewed as a transfinite resultant with two distinguished null-variables $u$ and $v$.
The $n$th approximant is the determinant of size $(2n+2)\times (2n+2)$:
\begin{equation}\label{Syl1}
\begin{split}
\mathcal{E}_{a,b}(u,v)_n
&=\frac{1}{(uv)^{-n-1}}
\det
\begin{pmatrix}
v   & -b_0  &  \ldots & -b_n \\
     & \ldots &\ldots &  \ldots & \ldots \\
    &  &v  & -b_0 & \ldots & -b_n   \\
   -a_n  & -a_{n-1} & \ldots & u \\
      & \ldots &\ldots &  \ldots & \ldots \\
     &  &  -a_n & -a_{n-1} &\ldots & u  \\
\end{pmatrix}.
\end{split}
\end{equation}
We shall see (see Remark~\ref{bbelow} below) that for $b=\bar a$, the transfinite elimination function is related to the exponential transform. In general the following analogue of elimination property holds.

\begin{proposition}\label{pro:res}
Let $f=\sum_{k=0}^\infty a_kz^{k+1}$ and $g=\sum_{k=0}^\infty b_kz^{-k-1}$. Then
\begin{equation*}
\label{consis}
\mathcal{E}_{a,b}(f,g)=0,
\end{equation*}
in the sense that
\begin{equation}
\label{consis1}
\mathcal{E}_{a,b}(f_n,g_n)_n=0, \quad n\geq 1,
\end{equation}
where $f_n$ and $g_n$ are the $n$th approximants of $f$ and $g$ respectively.

\end{proposition}

\begin{proof}
It suffices to prove (\ref{consis1}). To this aim notice that the resultant of the two polynomials
$$
\widetilde{g}(z)=(g_n(z)-v)z^{n+1}=b_n+b_{n-1}z+\ldots+b_0z^n-vz^{n+1},
$$
and
$$
\widetilde{f}(z)=f_n(z)-u=-u+a_0z+a_{1}z^2+\ldots+a_nz^{n+1}
$$
with respect to the variable $z$ coincides with  the determinant in (\ref{Syl1}):
$$
\mathcal{E}_{a,b}(u,v)_n=\frac{1}{(uv)^{n+1}}\Res_{\mathrm{pol}}(\widetilde{g},\widetilde{f}).
$$
By the characteristic property of the resultant,  $\mathcal{E}_{a,b}(u,v)_n$ vanishes if and only if $\widetilde{g}$ and $\widetilde{f}$ have a common root $z_0\in\Com{}$, i.e. $\widetilde{f}(z_0)=\widetilde{g}(z_0)=0$ for some complex $z_0$. This is equivalent to saying that
$g_n(z_0)-v=0$ and $f_n(z_0)-u=0$. Hence $\mathcal{E}_{a,b}(f_n,g_n)_n$ equals identically zero which yields the desired property.

\end{proof}


\section{Transfinite functions on closed analytic curves}

A general idea of a transfinite function on closed analytic curves is modeled on the following observation. Consider any parameterized curve $\Gamma=f_n(\mathbb{T})$, where $f_n$ is the $n$th approximant to the transfinite series (\ref{pol})
\begin{equation}\label{resulting}
f_n(z)=\sum_{k=0}^na_{k}z^{k+1}, \qquad a_0>0,
\end{equation}
and $\mathbb{T}$ is the unit circle.
Regarding the coefficients of $f_n(z)$ as a coordinate system on the space of parameterized curves $\Gamma$, many established functionals can be written as  functions of the coefficients $a_k$. Those functionals which are polynomials in $a_k$ for  $k\geq 1$ will be called \textit{admissible}.

Let $h: \Gamma\to \Com{}$ be any admissible functional and $h_n$ its resulting expression when $\Gamma$ has the form (\ref{resulting}). If the sequence $(h_n)$ satisfies the condition (\ref{restr}), it extends to a transfinite function $\widetilde{h}\in \Com{}_\infty[a_0,a_1,\overline{a}_1,\ldots]$ which is called the transfinite extension of $h$. In that case $\widetilde{h}$ can be thought of as the `value' of $h$ on the transfinite series
$
f(z)=\sum_{k=0}^{\infty}a_kz^{k+1}
$
which, in its turn, can be regarded as an `ideal'  curve.  Hence the formalism of transfinite functions can be applied to translate admissible functionals onto algebraic language. Below we demonstrate how this principle works with the Schwarz function as an example.


\subsection{The Schwarz function and harmonic moments}\label{sec:Sch}

We start with standard definitions. With any analytic
curve $\Gamma$ (not necessarily closed in general) one can
associate the Schwarz function, that is a holomorphic in a
neighborhood of  $\Gamma$  function  $S(\zeta)$ such that
\begin{equation}\label{elim}
S(\zeta)=\bar \zeta, \quad \zeta\in \Gamma.
\end{equation}
The above characteristic property is important when manipulating with the anti-holomorphic coordinate $\bar \zeta$
by substituting a \textit{holomorphic} function $S(\zeta)$. The domain of definition of the Schwarz
function is usually not a priori given, but one may always choose it
to be symmetric with respect to $\Gamma$. Alternatively, one may
think of $S(\zeta)$ only as a germ of an analytic function given on
the curve.

In the other extreme, there is one case with a maximally unsymmetric
domain of definition of the Schwarz function which singles out as
being particularly tractable and having a rich theory: this is when
$\Gamma=\partial\Omega$ for some domain $\Omega$ and the Schwarz
function extends to being a meromorphic function in all of $\Omega$.
Then $\Omega$ is called quadrature domain, or algebraic domain
\cite{Aharonov-Shapiro76}. It turns out that this requirement is
rigid enough (see \cite{Aharonov-Shapiro76}, \cite{Gustafsson83}) to
ensure that $S(\zeta)$ even is an algebraic function. Moreover, a
simply connected quadrature domain is an image of the unit disk
under a rational univalent function \cite{Aharonov-Shapiro76}.

Now assume that $\Gamma$ is a boundary of a simply connected domain
$\Omega$ containing the origin and denote its Schwarz function by $S(\zeta)$. The
following integrals make sense:
\begin{equation}\label{Sch20}
M_k(\Gamma)=\frac{1}{2\pi \I}\int_{\Gamma}\bar {\zeta} \zeta^{k} \;d\zeta,
\qquad k\in \mathbb{Z},
\end{equation}
which for nonnegative $k$ may be rewritten as
\begin{equation}\label{MGO}
M_k(\Gamma)=\frac{\I}{2\pi}\iint_{\Omega}\zeta^{k}\;
d\zeta\wedge d\bar \zeta, \qquad k\geq0.
\end{equation}
These quantities are known as harmonic or complex moments of the
domain $\Omega$. For negative $k$, (\ref{MGO}) still make sense as
principal value integrals. Alternatively, (\ref{Sch20}) may be turned
into area integrals by passing to the complement
$\Com{}\setminus\overline{\Omega}$.

In what follows we shall think of the harmonic moments as
generalized Laurent coefficients. More precisely, note that the
substitution of the definition of $S(\zeta)$ into (\ref{Sch20})
yields
\begin{equation}\label{Sch21}
M_k(\Gamma)=\frac{1}{2\pi \I}\int_{\Gamma}S(\zeta) \zeta^{k} \;d\zeta,
\qquad k\in \mathbb{Z},
\end{equation}
hence, one can think of $M_{-k}(\Gamma)$ as the  Laurent coefficients (with respect to $\Gamma$) of the shifted Schwarz function $\zeta S(\zeta)$. We shall write this as
\begin{equation}\label{precise}
S(\zeta)\sim \sum_{k\in \mathbb{Z}}M_k(\Gamma)\zeta^{-k-1}.
\end{equation}

It worth to notice that (\ref{Sch20}), as well as (\ref{Sch21}),
allows to define the moments for any parameterized curve. In a more
generality, this reduces to considering the moments of a function
rather than of a curve. Indeed, we recall that by the Riemann
mapping theorem simply connected domains (with the origin inside)
are in one-to-one correspondence with holomorphic and univalent in
the unit disk $\mathbb{D}$ functions $f(z)$ normalized by $f(0)=0$
and $f'(0)\in \mathbb{R}^+$.

Assume that the boundary of $\Omega$ is an analytic curve.
Then a uniformizing function $f$ may be chosen to be holomorphic
in the closed unit disk. Hence substituting of $\zeta=f(z)$ into
(\ref{Sch20}) gives
\begin{equation}\label{Sch22}
\mu_k(f):=M_k(f(\mathbb{T}))=\frac{1}{2\pi \I}\int_{\Gamma}
f^*(z) f^k(z) f'(z) dz, \qquad k\in \mathbb{Z},
\end{equation}
where $\mathbb{T}=\partial \mathbb{D}$ and a holomorphic in
$\Com{}\setminus\mathbb{D}$ function $f^*$ is defined by
$$
f^*(z)=\overline{f(1/\bar{z})}.
$$
It is natural to refer to $\mu_k(f)$ as the moments of the function
$f$.

\begin{proposition}\label{pr00}
Let $n\geq 1$ and let $f_n(z)$ be polynomial (\ref{resulting}) such that $f_n(z)z^{-1}$ has no zeros in the closed
unit disk. Then the $\mu_k(f_n)$ are rational functions of the
coefficients of $f_n$. Moreover
\begin{equation*}\label{anan}
a_0^{n-2k-1}\mu_{k}(f_n)\in \mathbb{C}[a_0,a_1,\overline{a}_1,\ldots,a_n,
\overline{a}_n], \qquad \forall k\in \mathbb{Z}
\end{equation*}
and
\begin{equation}\label{ii}
\mu_k(f_n)=0, \qquad \forall k\geq  n+1
\end{equation}
In particular, $\mu_k(f_n)$ are admissible.
\end{proposition}

\begin{proof}
One can rewrite  (\ref{Sch22}) as follows
\begin{equation*}\label{extra}
\mu_k(f_n)=\mathrm{CT}_z(zf_n^*\,  f'_n\cdot f_n^k),
\end{equation*}
where $\mathrm{CT}_z$ denotes constant term extraction (with respect
to $z$). Then proposition follows from the fact that $zf_n'f_n^*$
contains only the terms $z^{m}$ with $-n\leq m\leq n$ and from the
Laurent expansion
\begin{equation*}\label{simple}
\begin{split}
f_n^k(z)&=a_0^{k}z^{k}(1+\frac{1}{a_0}\sum_{i=1}^na_iz^i)^k\\
&=a_0^{k}z^{k}\sum_{j=0}^\infty \frac{k(k-1)\ldots(k-j+1)}{j!}
\biggl(\sum_{i=1}^n\frac{a_iz^i}{a_0}\biggr)^j.\\
\end{split}
\end{equation*}
\end{proof}
It is a remarkable property of the harmonic moments that they
can be extended as transfinite elements in the sense described in the beginning of this section. Indeed, a simple analysis of the above series for $f_n^k$ shows that
\begin{equation}\label{imme}
\mu_k(f_{n})|_{a_{n}=\overline{a}_{n}=0}=\mu_k(f_{n-1}),
\end{equation}
hence the following projective limit exists:
$$
\mu_k(f)=\lim_{\leftarrow n}\mu_k(f_n)\in \Com{}_\infty[a_0,a_1,\overline{a}_1,\ldots]
$$
and will be called the $k$th transfinite harmonic moment. The zero
moment $\mu_{0}(f)$ is found to be
\begin{equation*}\label{M0}
\mu_0=a_0^2+2a_1\bar{a}_{1}+3a_2\bar{a}_{2}+\ldots,
\end{equation*}
and equals the normalized area of the domain $f_n(\mathbb{D})$ when
$f_n$ is a univalent function. Clearly $\mu_0$ is positive
and real (i.e., $\mu_0^*=\mu_0$).

The formulas for higher moments, especially for the negative ones,
are much more involved. The first
approximants for $k=1$ and $k=-1$ are
\begin{equation*}\label{m1}
\begin{split}
\mu_{1}(f_1)&=a_0^2 \bar{a}_{1},\\
\mu_{1}(f_2)&=a_0^2 \bar{a}_{1}+3a_0a_{1}\bar{a}_{2},\\
\mu_{1}(f_3)&=a_0^2 \bar{a}_{1}+3a_0a_{1}\bar{a}_{2}+4a_0a_{2}
\bar{a}_{3}+2a_1^2\bar{a}_{3},\\
\end{split}
\end{equation*}
and
\begin{equation*}\label{m2}
\begin{split}
\mu_{-1}(f_1)&=a_1-\frac{a_1^2 \bar{a}_{1}}{a_0^2},\\
\mu_{-1}(f_2)&=a_1+\frac{2a_2\bar{a}_{1}}{a_0}-\frac{a_1^2
\bar{a}_{1}+3a_{2}^2\bar{a}_{2}}{a_0^2}+\frac{a_1^3\bar{a}_{2}}{a_0^4}.\\
\end{split}
\end{equation*}

In general, for non-negative values of $k$ Richardson's formula
\cite{R72}
\begin{equation*}
\mu_k(f_n)=\sum (s_0+1) a_{s_0}\cdots a_{s_{k}}\bar{ a}_{s_0+\ldots
+s_{k}+k}, \label{summa}
\end{equation*}
holds, where the summation is  over all multiindices
$(s_0,\ldots,s_k)$, $0\leq s_j\leq n$.

\subsection{Harmonic moments via resultant}\label{below}
Here we  describe briefly another way  to obtain the harmonic moments with non-negative indices. Given a bounded domain $\Omega$, the function of two complex variables defined by
\begin{equation*}
\exp[\frac{1}{2\pi\I}\int_\Omega
\frac{d\zeta}{\zeta -z}\wedge \frac{d\bar{\zeta}}{\bar{\zeta} -\bar {w}}]=:E_\Omega (z,w) :
\;(\Com{}\setminus \overline{\Omega})^2\to \Com{}
\end{equation*}
is called the exponential transform of the domain $\Omega$ (see, e.g., \cite{Carey-Pincus74}, \cite{P98},
\cite{GP98}). Expanding the integral
in power series in  $1/\bar w$ gives
\begin{equation}\label{Cauchy0}
E_\Omega(z,w)=1-\frac{1}{\bar w}C_\Omega(z)+ \mathcal{O}(\frac{1}{|w|^2})
\end{equation}
as $|w|\to \infty$, with $z\in\mathbb{C}\setminus \overline{\Omega}$ fixed. Here
\begin{equation}\label{when}
C_\Omega(z)= \frac{1}{2\pi \I}\int\limits_{\Omega} \frac{d\zeta\wedge d\bar\zeta }{z-\zeta}=\sum_{k\geq 0}\frac{M_k(\Gamma)}{z^{k+1}}, \quad \text{as $z\to \infty$},
\end{equation}
is the Cauchy transform of $\Omega$, and $M_k(\Gamma)$ are defined as in (\ref{MGO}).

When $\Omega=f_n(\mathbb{D})$, where $f_n$ is a univalent in the closed unit disk polynomial (\ref{resulting}), the sum in (\ref{when}) contains only terms with degrees $k\leq n$ and (\ref{Cauchy0}) becomes
$$
E_{f_n(\mathbb{D})}(z,w)=1-\frac{1}{\bar w}\sum_{k=0}^n\frac{\mu_k(f_n)}{z^{k+1}}+ \mathcal{O}(\frac{1}{|w|^2}).
$$
On the other hand, the authors showed in \cite{GT07} that the exponential transform of such $f_n(\mathbb{D})$ is the meromorphic resultant:
\begin{equation}\label{Exp1}
E_{f_n(\mathbb{D})}(z,w)=\Res_\zeta(-z+\sum_{k=0}^n a_k\zeta^{k+1},-\bar w+\sum_{k=0}^n \bar a_k\zeta^{-k-1})\\
\end{equation}
Combing these formulas we obtain
$$
\det
\begin{pmatrix}
  1      &       &       &  -\frac{\bar a_n}{\bar w}  &        &\\
  -\frac{a_0}{z}   &\ddots &       & \vdots             & \ddots       &\\
  \vdots &       &  1    &  -\frac{\bar a_0}{\bar w}  &  & -\frac{\bar a_n}{\bar w}\\
  -\frac{a_n}{z}   &       & -\frac{a_0}{z}  & 1                 &        &\vdots    \\
         & \ddots& \vdots&                    &\ddots        & -\frac{\bar a_0}{\bar w}   \\
         &       & -\frac{a_n}{z}  &                    &  & 1\\
\end{pmatrix}
=1-\frac{1}{\bar w}\sum_{k=0}^n\frac{\mu_k(f_n)}{z^{k+1}}+ \mathcal{O}(\frac{1}{|w|^2}).
$$
Hence the above determinant completely determines all the harmonic moments $\mu_k(f_n)$ for $0\leq k\leq n$ and expanding the determinant in $\bar w$, one gets explicit formulas.

\begin{remark}\label{bbelow}
Another corollary of (\ref{Exp1}) is that $E_{f_n(\mathbb{D})}(z,w)$ coincides with the $n$th approximant of the transfinite elimination function (\ref{Syl1}) for  $a=(a_{k})_{k\geq 0}$ and $b=(\bar a_k)_{k\geq 0}$. This can be thought as a transfinite analogue of the coincidence of the elimination function and the exponential transform on the level of transfinite functions.
\end{remark}

\subsection{The transfinite Schwarz function}
By property (\ref{ii}) in Proposition~\ref{pr00} the series
$$
S(f_n,\zeta)=\sum_{k\in \mathbb{Z}}\mu_k(f_n)\zeta^{-k-1}
$$
contains only finite number of negative terms, hence may be
interpreted as formal Laurent series with coefficients in
$\mathbb{C}[a_0,a_1,\overline{a}_1,\ldots,a_n,\overline{a}_n]$. It
follows immediately from (\ref{imme}) that $S(f_n,\zeta)$ also
satisfies the transfinity condition. Hence
$$
S_f(\zeta)=\lim_{\leftarrow n}\sum_{k\in \mathbb{Z}}\mu_k(f_n)\zeta^{-k-1}
$$
is well-defined on the level of formal Laurent series and in this setting  formula (\ref{precise}) makes a rigorous sense.
Moreover, the characteristic property (\ref{elim}) of the Schwarz function
reads in the new notations as follows:
\begin{equation*}\label{11}
S_f(f)=f^*.
\end{equation*}


\section{A generalized moment map}

Let a closed Jordan analytic curve be given. The most natural choice
of a coordinate system for an analytic curve is the coefficients of
the uniformization map, which maps the unit disk onto the interior
of the curve.

Another  choice comes from the Schwarz function $S(\zeta)$ of the curve which by
(\ref{elim}) contains complete information about the curve. Therefore
the harmonic moments may be thought as coordinates.
However one can extract information about the curve from harmonic
moments in many different ways. We shall consider the following two.
Let $f_n$ is given by (\ref{resulting}). Taking into account that
$\mu_0$ and $a_0$ are real (in fact, positive), we define the so-called \textit{complete moment map} by
\begin{equation}\label{mm1}
\mu(\bar a_{n},\ldots,\bar a_{1},a_0,\ldots,a_n)=
({\mu_{-n}},\ldots,\mu_0,\ldots,\mu_n):\; \R{}\times
\mathbb{C}^{2n}\to \R{}\times \mathbb{C}^{2n},
\end{equation}
where $\mu_k=\mu_k(f_n)$ (recall that $\mu_k(f_n)=0$ for $k\geq n+1$).
It is well-known that the right half of the latter map (the moments with non-negative indices)
determines $f_n$ at least locally and usually one takes it as an alternative (to the coefficients $a_k$)
set of coordinates (see, e.g., \cite{KKMWZ}).

Another moment map, which was treated recently in
\cite{KT}, \cite{T2005}, is a moment map consisting of the
nonnegative moments and their conjugates:
\begin{equation}\label{mm2}
\mu^*(\bar a_{n},\ldots,\bar a_{1},a_0,\ldots,a_n)=
({\bar\mu_{n}},\ldots,\bar \mu_{1},\mu_0,\ldots,\mu_n):\;
\R{}\times \mathbb{C}^{2n}\to \R{}\times \mathbb{C}^{2n}.
\end{equation}

These two maps may be written in a common form
\begin{equation}\label{muph}
\phi(\bar a_{n},\ldots,\bar a_{1},a_0,\ldots,a_n)
=(\phi_{-n},\ldots,\phi_{-1},\phi_0,\ldots,\phi_n),
\end{equation}
with the generalized moments $\phi_k$ given by the integrals
\begin{equation}\label{phii}
\phi_k=\frac{1}{2\pi \I}\int_{\mathbb{T}}\Phi_k(f_n,f_n^*)f_n'dz
=\mathrm{CT}_{z}(zf_n'\Phi_k(f_n,{f_n}^*)),
\end{equation}
where $\Phi_k(\zeta,\bar \zeta)$ are suitable functions. If
$f:\mathbb{D}\to \Omega$ is a uniformizing map of a simply
connected domain $\Omega$ then these moments are
$$
\phi_k=\frac{1}{2\pi\I }\iint_{\Omega}\frac{\partial \Phi_k}{\partial \overline{\zeta}}\;d\zeta\wedge d \overline{\zeta}=
\frac{1}{2\pi\I }\int_{\Gamma}\Phi_k\;d\zeta.
$$
For the complete moment map (\ref{mm1})
\begin{equation*}\label{form1}
\Phi_k(\zeta,\bar \zeta)\equiv \zeta^k\bar \zeta,\quad (k\in \mathbb{Z})
\end{equation*}
and it is not hard to check that for (\ref{mm2})
\begin{equation*}\label{form2}
\Phi_k(\zeta,\bar \zeta)=\left\{
\begin{array}{ll}
\zeta^k\bar \zeta & \text{if $k\geq0$,}\\
\frac{1}{1-k}{\bar\zeta}^{1-k}& \text{if $k\leq -1$,}\\
\end{array}
\right.
\end{equation*}

Our main result below shows that the Jacobian of the generalized
moment map $\phi$ always splits into two distinguished factors: the
first depends on a concrete form of the functions $\Phi_k$ and the
second is the self-resultant of the derivative $f'$. To formulate
it, it is convenient to set
$$
a_{-n}=\bar{a}_n.
$$

\begin{theorem}
\label{th:main} Let $n\geq 1$ and $\Phi_k(\zeta,\bar \zeta)$,
$-n\leq k\le n$, be a system of rational functions. For $f_n(z)=z
\sum_{k=0}^na_kz^k$ introduce the residue matrix
\begin{equation}\label{givven}
v_{kj}=\frac{1}{2\pi \I}\int_{\mathbb{T}}\Phi_k(f_n,{f_n}^*)\frac{dz}{z^{1+j}}.
\end{equation}
Then the Jacobian of the generalized moment map (\ref{muph}) is
\begin{equation*}\label{rm1}
\begin{split}
\frac{\partial \phi}{\partial a}&\equiv
\frac{\partial (\phi_{-n},\ldots,\phi_0,\ldots,\phi_n)}
{\partial (a_{-n},\ldots,a_0,\ldots,a_n)}=2a_0^{2n+1}
\det (v_{kj})\cdot \Res(f_n',f_n'^*).\\
\end{split}
\end{equation*}
Here and in what follows
we denote by $(r_{kj})$ the $(2n+1)\times (2n+1)$-matrix
with entries $r_{kj}$ with indices $k,j$ running interval between $-n$ and $n$.

\end{theorem}

In the two cases discussed above,
the corresponding matrices $(v_{kj})$ are easily
found to be upper diagonal with $a_0$ to  certain degrees
on the diagonal. This yields the following corollary.

\begin{corollary}
In the introduced above notation, we have for the Jacobians:
$$
\frac{\partial (\mu_{-n},\ldots,\mu_0,\ldots,\mu_n)}{\partial (a_{-n},\ldots,a_0,\ldots,a_n)}=2a_0^{2n+1}\Res(f_n',f_n'^*),
$$
and
$$
\frac{\partial (\bar \mu_{n},\ldots,\bar \mu_1,\mu_0,\ldots,\mu_n)}{\partial (a_{-n},\ldots,a_0,\ldots,a_n)}=2a_0^{n^2+3n+1}\Res(f_n',f_n'^*),
$$
In particular, the transition Jacobian between these two maps is given by
\begin{equation*}\label{sled}
\frac{\partial (\bar \mu_{n},\ldots,\bar \mu_1,\mu_0,\ldots,\mu_n)}
{\partial (\mu_{-n},\ldots,\mu_{-1},\mu_0,\ldots,\mu_n)}=a_0^{n^2+n}.
\end{equation*}
\end{corollary}

\begin{proof}[Proof of Theorem~\ref{th:main}]

Let $u(z)$ represent a direction for variation of $f_n(z)$ and let $d\phi_k(u)$ denote the directional derivative of $\phi_k$ taken at  $f_n$ along the function $u$. Then by virtue (\ref{phii}) we have
\begin{equation}\label{sss1}
\begin{split}
d\phi_k(u)&=\lim_{t\to 0}\frac{\phi_k(f+tu)-\phi_k(f)}{t}\\
&=\mathrm{CT}_{z}[zuf'\partial_\zeta \Phi_k(f,f^*)+ zu^*f'\partial_{\bar\zeta} \Phi_k(f,f^*)+ zu'\Phi_k(f,f^*)].
\end{split}
\end{equation}
Integrating by parts we find for the last term in (\ref{sss1}):
\begin{equation*}
\begin{split}
\mathrm{CT}_{z}[zu'\Phi_k(f,f^*)]&
=\frac{1}{2\pi \I}\int_{\mathbb{T}}\Phi_k(f,f^*)\; du\\
&=-\frac{1}{2\pi \I}\int_{\mathbb{T}}(uf'\partial_{\zeta} \Phi_k(f,f^*)+u{f^*}'\partial_{\bar\zeta} \Phi_k(f,f^*))dz\\
&=\mathrm{CT}_{z}[-zuf'\partial_{\zeta} \Phi_k(f,f^*) +\frac{1}{z}u{f'}^*\partial_{\bar\zeta} \Phi_k(f,f^*)  \biggr],
\end{split}
\end{equation*}
where we used
$$
(f^*(z))'=-\frac{1}{z^2}f'^*(z).
$$
Substitution of this into (\ref{sss1}) gives
\begin{equation*}\label{sss2}
\begin{split}
d\phi_k(u)&=\mathrm{CT}_{z}[(zf'u^*+\frac{1}{z}f'^*u)\cdot \partial_{\bar\zeta} \Phi_k(f,f^*)],
\end{split}
\end{equation*}
and setting $h(z)=\frac{u(z)}{z}=h_0+h_1z+\ldots+h_nz^n$
we arrive at the following  formula:
\begin{equation}\label{sss3}
\begin{split}
d\phi_k(zh)&=\mathrm{CT}_{z}[(f'h^*+f'^*h)\cdot \partial_{\bar\zeta} \Phi_k(f,f^*)]\equiv
\sum_{j=-n}^n \phi_{kj}h_j,
\end{split}
\end{equation}
where $h_{-k}=\bar h_k$, $k\geq 1$. In this notation the required
Jacobian spells out as
\begin{equation}\label{found}
\frac{\partial \phi}{\partial a}=\det (\phi_{kj})_{-n\leq k,j\leq n}.
\end{equation}
Using (\ref{givven}) we find from  (\ref{sss3})
$$
\sum_{j=-n}^n \phi_{kj}h_j=\sum_{i,j=-n}^n v_{ki}u_{ij}h_j
$$
where
\begin{equation}\label{need}
\mathrm{CT}_{z}[(f' h^*+f'^*h)z^i]=\sum_{j=-n}^n u_{ij}h_j.
\end{equation}
This yields by virtue of (\ref{found})
\begin{equation*}\label{found1}
\frac{\partial \phi}{\partial a}=\det (v_{ki}) \cdot \det (u_{ij}).
\end{equation*}
Hence  we only need  to find the determinant of $U=(u_{jm})$. Let us write
$$
f'(z)=\sum_{k=0}^nb_kz^k, \qquad f'^*(z)=\sum_{k=0}^nb_{-k}z^{-k},
$$
where $b_k=(k+1)a_k$ and $b_{-k}=(k+1)\bar a_k$  for $k\geq 0$ (notice that this index notation is consistent with $b_0=a_0\in \mathbb{R}$). Then an explicit form of the matrix $U$ is easily found from (\ref{need}):
$$
U=\begin{pmatrix}
b_0 &  &  & & b_{-n} & &   &  \\
b_1 & b_0 &  &  & b_{1-n} &  b_{-n} &    &  \\
\vdots & \vdots & \ddots &  & \vdots & \vdots & \ddots  &  & \\
b_{n-1} & b_{n-2} & \ldots & b_{0} & b_{-1} & b_{-2} &  \ldots & q_{n-1} &   \\
b_{n} & b_{n-1} & \ldots & b_{1} & 2b_0 & b_{-1} &  \ldots & b_{1-n} & b_{-n}  \\
& b_{n} & \ldots & b_2 & b_1 & b_0  &  \ldots & b_{1-n} & b_{1-n} \\
& & \ddots & \vdots &  \vdots & & \ddots  &  \vdots & \vdots  \\
& & & b_{n} &b_{n-1} &    & & b_{0} & b_{-1}  \\
& & &  &b_{n} &  &   & & b_0  \\
\end{pmatrix}.
$$

Denote by $U_{k}$ the $k$th column $(u_{ik})_{i=-n}^n$ in $U$. Then
\begin{equation}\label{subss2}
b_0U_0+\sum_{i=1}^{n}b_{-i}U_{-i}-\sum_{i=1}^{n}b_{i}U_i=2b_0Z,
\end{equation}
where the column vector  $Z$ has the form
$$
Z=(b_{-n},\ldots,b_{-1},b_0,0,\ldots,0)^\top
$$
with the last $n$ entries equal to zero. It follows then from (\ref{subss2})  that
\begin{equation*}
\begin{split}
\det U&=2\det
\begin{pmatrix}
b_0 &  &  & & b_{-n} & &   &  \\
b_1 & b_0 &  &  & b_{1-n} &  b_{-n} &    &  \\
\vdots & \vdots & \ddots &  &\vdots  & \vdots & \ddots  &  & \\
b_{n-1} & b_{n-2} & \ldots & b_{0} & b_{-1} & b_{-2} &  \ldots & b_{-n} &   \\
b_{n} & b_{n-1} & \ldots & b_{1} & b_0 & b_{-1} &  \ldots & b_{1-n} & b_{-n}  \\
& b_{n} & \ldots & b_2 & & b_0  &  \ldots & b_{2-n} & b_{1-n} \\
& & \ddots & \vdots &  & & \ddots  &  \vdots & \vdots  \\
& & & b_{n} & &    & & b_{0} & b_{-1}  \\
& & &  & &  &   & & b_0  \\
\end{pmatrix}\\
\end{split}
\end{equation*}
Now expanding the latter determinant by the last row and taking into account (\ref{formula}), we get
$$
\det U=2b_0^{2n+1} \Res(\sum_{k=0}^n b_kz^{k},\sum_{k=0}^n b_{-k}z^{-k})
=2a_0^{2n+1}\Res(f',f'^*),
$$
which finishes the proof.
\end{proof}



\end{document}